\documentclass[12pt,a4paper]{article}

\usepackage[T1]{fontenc}
\usepackage[utf8]{inputenc}
\usepackage[a4paper,margin=1in]{geometry}
\usepackage{microtype}
\usepackage{parskip}
\usepackage{enumitem}
\usepackage{lmodern}
\usepackage{graphicx}
\usepackage{xcolor}

\usepackage{amsmath,amssymb,amsfonts,amsthm,mathtools}
\usepackage{mathrsfs}
\usepackage{bm}

\usepackage[colorlinks=true,linkcolor=blue,citecolor=blue,urlcolor=blue]{hyperref}

\newtheorem{introtheorem}{Theorem}

\newtheorem{theorem}{Theorem}[section]
\newtheorem{lemma}[theorem]{Lemma}
\newtheorem{proposition}[theorem]{Proposition}
\newtheorem{corollary}[theorem]{Corollary}

\theoremstyle{definition}
\newtheorem{definition}[theorem]{Definition}

\theoremstyle{remark}

\newcommand{\ip}[2]{\left\langle #1,#2\right\rangle}
\newcommand{\R}{\mathbb{R}}

\newcommand{\Sp}{\mathbb{S}}

\newcommand{\cA}{\mathcal{A}}

\DeclareMathOperator{\Span}{span}

\newcommand{\E}{\mathbb{E}}

\newcommand{\norm}[1]{\left\lVert#1\right\rVert}

\newcommand{\dd}{\,\mathrm{d}}

\title{Aomoto interpolation and Coxeter systems}

\author{
	Ángel D. Martínez
	\and
	Oscar Ortega-Moreno
}

\begin{document}
	
	\maketitle
	
	\begin{abstract}
	In this paper, we construct a Lagrange-type basis for the Aomoto space $AO(\mathcal A)$, naturally indexed by the chambers of the hyperplane arrangement $\mathcal A$. The construction relies on a dimension theorem of Orlik and Terao and yields an interpolation formula for elements of $AO(\mathcal A)$. We use this formula to characterize the extremal configurations in the strong polarization inequality as those arising from finite Coxeter reflection systems. We further show that the interpolation formula gives rise to a family of \emph{chamber identities}, including identities that were central to our earlier proof of the strong polarization problem and the Gaussian product inequality. Finally, we adapt the recent breakthrough of Ouimet and Greaves to prove a generalized Gaussian Product Inequality for completely monotone functions.
\end{abstract}
	
	\section{Introduction}
	\label{sec:intro}
	
	Given the central arrangement of hyperplanes
	$\cA=\{v_1^{\perp},\dots,v_n^{\perp}\}$, the space of rational functions spanned
	by the reciprocals of products of the linear forms $\langle v_j,x\rangle$ taken over linearly independent subsets of the $v_j$ is known as the \emph{Aomoto space} $AO(\cA)$ (see Section \ref{sec:chamber-basis} for more details). In 1992 Aomoto conjectured, and Orlik and Terao subsequently proved, the following formula for the dimension of this vector space:
	\begin{theorem}[Orlik and Terao]\label{thm:orlikterao}
		The dimension of the Aomoto space equals the number of chambers of the hyperplane arrangement	$$
		\dim AO(\cA)=\#\{\text{chambers of }\cA\}.
		$$
	\end{theorem}
	
We refer to \cite{orlik-terao} for a proof and more details. In our earlier work on polarization problems~\cite{MOM}, a central role was played by the polynomial
	\begin{equation}
		P(x)=\prod_{j=1}^{n}\langle v_j,x\rangle
		\label{eq:P}
	\end{equation}
	and by the rational functions derived from it.  There we studied the set $\mathscr{E}(P)$ of local extrema of $P$ on the unit sphere, one in each chamber of $\cA$, and the weighted identities these extrema satisfy were decisive in the argument (cf.\ Theorem~\ref{thm:weigthed} below). The present paper starts from the observation that the rational functions appearing in that work belong naturally to the Aomoto space $AO(\cA)$.
	
	To each local extremum the present paper attaches an explicit rational function in $AO(\cA)$. Its definition uses a positive weight introduced in~\cite{MOM}: for every $u\in\mathscr{E}(P)$, set
	$$
	\mu(u)
	=
	\det\left(
	I+\frac1n\sum_{j=1}^{n}\frac{v_j\otimes v_j}{\langle v_j,u\rangle^{2}}
	\right)^{-1},
	$$
	where $v\otimes v$ denotes the rank-one operator
	$$
	(v\otimes v)z=\langle v,z\rangle v .
	$$
	The tensor inside the determinant is positive definite, so $\mu(u)>0$.
	\begin{definition}
		For $u\in\mathscr{E}(P)$, define the rational function
		$$
		L_u(x)
		=
		\mu(u)\,
		\det\left(
		I+\frac1n\sum_{j=1}^{n}
		\frac{v_j\otimes v_j}{\langle v_j,u\rangle\,\langle v_j,x\rangle}
		\right).
		$$
	\end{definition}
	We show below that each $L_u$ lies in $AO(\cA)$, and that these functions form a basis of $AO(\cA)$ enjoying a Lagrange interpolation property.
    
	\begin{introtheorem}[Interpolation formula]\label{thm:introA}
		The functions $\{L_u:u\in\mathscr{E}(P)\}$ form a basis of $AO(\cA)$, and every $f\in AO(\cA)$ satisfies
		$$
		f(x)=\sum_{u\in\mathscr{E}(P)}f(u)\,L_u(x).
		$$
	\end{introtheorem}
	
	This is restated and proved as Theorem~\ref{thm:basis} below. The formula translates
	identities in $AO(\cA)$ into finite sums over the critical points; specializing
	it to the constant function and to the individual reciprocals
	$1/\langle v_i,x\rangle$ yields a family of \emph{chamber identities} for the
	weights $\mu(u)$. These are collected in Theorems~\ref{thm:chamber-identities} and~\ref{thm:weigthed}. Of particular interest among them is identity~\eqref{eq:local-global}, on which the application developed below rests.
	
	The same identity also underlies a recent proof of the Gaussian Product Inequality~\cite{OUI}, which is closely connected to the mechanism introduced in our earlier work on polarization problems and Coxeter systems. Its central algebraic step is not independent of ours: both attach positive determinant weights to distinguished points in the chambers, or branches, and establish the same weighted-average identities over them. Under the natural identification of the correlation matrix with the Gram matrix of the defining vectors, the branch weights and identities in the Gaussian proof coincide with those in our proof of the strong polarization theorem. The novelty of that proof lies in an ingenious change of variables which, combined with a representation formula and an application of Jensen's inequality, yields the result. 
    
    It should be noted, however, that their derivation of the algebraic identities requires the hyperplanes to divide space into the maximal number of chambers. Our original derivation, by contrast, rested on the classical Euler--Jacobi vanishing theorem and applied to more general arrangements. In the present paper, using Theorem~\ref{thm:introA}, we give a new proof of these identities. Our two approaches appear to be linked through the theory of residues and the interpolation properties (that seem to have motivated the original discovery of Euler and Jacobi).
	
	It seems then, that our interpolation theorem uncovers and substantially generalizes the structure underlying this common mechanism: it applies to arbitrary central hyperplane arrangements, allowing linear dependencies among the defining vectors, which is precisely the generality afforded by Theorem~\ref{thm:introA}. The main application we develop below concerns the strong polarization inequality~\cite{MOM}, which we now recall. A second application can be found in Section \ref{sec:revisited}, where we introduce a generalized Gaussian Product Inequality that can be proved following the recent breakthrough of Ouimet and Greaves (cf. Theorem~\ref{prop:main} and its Corollary~\ref{corr} below).
	
	\begin{theorem}[Strong polarization conjecture, \cite{MOM}]\label{conj:strong-pol}
		For any unit vectors $v_1,\dots,v_n\in\Sp^{d-1}$ there exists a unit vector
		$u\in\Sp^{d-1}$ such that
		\begin{equation}\label{eq:strong-pol}
			\sum_{j=1}^{n}\frac{1}{\langle v_j,u\rangle^{2}}\le n^{2}.
		\end{equation}
	\end{theorem}
	
	The bound in \eqref{eq:strong-pol} is sharp. Let us note in passing that this result is also a consequence of Theorem~\ref{thm:weigthed}   below, which provides an alternative route to the main identity in \cite{MOM}. It is attained by every orthonormal
	basis, and, as Ball and Frenkel observed, already by $n$ equally spaced
	directions on a great circle; the problem therefore encodes geometric structure
	beyond that of orthogonal configurations. We refer to \cite{MOM} for more details and connections to other relevant results.
	
	In this paper we complete the characterization of its extremizers, confirming our suspicions from \cite{MOM}.  We call $v_1,\dots,v_n$ an \emph{extremal configuration} if
	$$
	\sum_{j=1}^n\frac{1}{\langle v_j,u\rangle^{2}}\ge n^{2}
	$$
	for every $u\in\Sp^{d-1}\setminus\{P=0\}$, so that the sharp bound is attained.
	Interpolating the quantity $\Delta P/P$ against the chamber basis linearizes the
	analysis of equality and gives the following rigidity statement.
	
	\begin{introtheorem}\label{thm:introB}
		If $v_1,\dots,v_n$ is an extremal configuration for \eqref{eq:strong-pol}, then
		the product
		$$P(x)=\prod_{j=1}^{n}\langle v_j,x\rangle
		$$
		is harmonic.
	\end{introtheorem}
	
	Harmonicity is a strong constraint (cf. \cite{Agranovsky2000}). Combined with the description of the extremal
	cases in~\cite{MOM}, it identifies the extremizers as \emph{finite reflection
		systems}: finite sets $\Phi\subset\Sp^{d-1}$ invariant under the reflections
	$s_v\colon x\mapsto x-2\langle v,x\rangle v$ across the hyperplanes $v^{\perp}$,
	$v\in\Phi$. The orthonormal bases and the regular planar configurations noticed
	by Ball and Frenkel are then the simplest members of a much larger family
	governed by Coxeter symmetry. Theorem~\ref{thm:introB} is a consequence of the very similar Proposition~\ref{thm:harmonicity}, and the resulting description of the equality case is given in Theorem~\ref{cor:equality} below.
	
	\begin{introtheorem}[Equality cases]\label{cor:equality}
		Let $v_1,\dots,v_n\in\Sp^{d-1}$ be unit vectors. Then
		$$
		\min_{x\in\Sp^{d-1}}
		\sum_{j=1}^{n}\frac{1}{\langle v_j,x\rangle^{2}}=n^{2}
		$$
		if and only if
		$
		\Phi=\{\pm v_1,\dots,\pm v_n\}
		$
		is a finite reflection system.
	\end{introtheorem}

This is an immediate consequence of our previous work and Theorem \ref{thm:introB}. Indeed, by Lemma 7.6 in~\cite{MOM} the vectors are non parallel for any extremal configuration. Proposition~\ref{thm:harmonicity} shows that $P$ is harmonic.  The characterization follows using the observation in \cite{MOM} together with Agranovsky's Theorem (cf. Theorem 6.1 
     \cite{Agranovsky2000}).

\subsection{Gaussian vector inequalities: revisited}\label{sec:revisited}

The strong polarization problem was known to imply the real polarization problem. On the other hand, another conjecture by Frenkel, known as the Gaussian Product Inequality was also known to imply the real polarization problem in the limiting case. This has recently been settled in \cite{OUI}. The main result of this section  explores the connection that brings  the two stronger conjectures together. Our approach isolates the general principle behind their proof, avoiding computations with special functions, applies to general hyperplane arrangements and avoids a couple of technicalities including the appearance of the correlation matrix all around. 

Before we proceed to state the main result we need to refresh some well-known results and introduce some notation.

\begin{definition}
A smooth function \(f:(0,\infty)\to\mathbb{R}\) is called
\emph{completely monotone}  and
\[
    (-1)^n f^{(n)}(x)\ge 0
\]
for every \(x>0\) and every integer \(n\ge 0\).
\end{definition}

Let $Z$ denote a one-dimensional standard Gaussian random variable and let
$\mathbf Z$ denote a standard Gaussian vector in a finite-dimensional Euclidean
space.  Only the vector-valued variables $\mathbf Z$ and $\mathbf t$ are written in boldface. As we shall explain later, the following result contains the Gaussian product inequality as a particular case.

Recall that a function $f:(0,\infty)\to[0,\infty)$ is called \emph{completely monotone} if $f\in C^\infty(0,\infty)$ and$$(-1)^k f^{(k)}(s)\geq 0$$for every $s>0$ and every integer $k\geq 0$.

\begin{introtheorem}[Gaussian Product Inequality for completely monotone functions]
\label{prop:main}
Let $v_1~,~\ldots~,~v_n~\in~\mathbb S^{d-1}$ be unit vectors, $\mathbf Z$ be a standard Gaussian vector in $\mathbb R^d$, and let $f_1~,~\ldots~,~f_n~:~(0,\infty)~\to~[0,\infty)$ be nonzero completely monotone functions. Then
$$
\mathbb E_{\mathbf Z}
\left(\prod_{j=1}^n
g_j\left(\frac{1}{\langle v_j,\mathbf Z\rangle^2}\right)\right)
\geq
\prod_{j=1}^n
\mathbb E_Z
g_j\left(\frac{1}{Z^2}\right),
$$
where $Z$ is a standard Gaussian random variable and the inequality is understood in $(0,\infty]$.
\end{introtheorem}

 Our proof is a variant of theirs with the aid of our analysis based on Aomoto spaces (cf. Lemma \ref{thm:chamber-identities2} below). The statement of Theorem \ref{prop:main} although aesthetic might be too cryptical. For that purpose let us deduce a few particular cases including the recent proof of the Gaussian Product Inequality by Ouimet and Greaves in \cite{OUI}. Upon inspection of their proof it seem natural to the authors to try to follow it upon the hypothesis on the functions involved to be Laplace transforms of some sort.

The Gaussian Product Inequality by taking
$$
f_j(s)=s^{-\alpha_j}.
$$
More generally, by Bernstein-Hausdorff-Widder's theorem, one may take
$$
f_j(s)=\int_{0}^{\infty}e^{-ts}d\mu_j(t)
$$
for arbitrary nonzero positive measures $\mu_j$ for which the integral is finite for every $s>0$. For another curious product inequality we refer to Corollary \ref{corr} below.

	\paragraph{Organization of the paper:} Section~\ref{sec:chamber-basis} introduces the Aomoto space and some of the notation and auxiliary properties we shall need. We then prove Theorem~\ref{thm:introA}, which contains the interpolation formula, as Theorem~\ref{thm:basis} below. In that section we state the chamber identities in Theorem~\ref{thm:chamber-identities}. These are then used to derive Theorem~\ref{thm:weigthed}, which includes an identity that was crucial in our previous work \cite{MOM} and in \cite{OUI}.
	Section~\ref{sec:equality} is devoted to the proof of Theorem \ref{thm:introB}. Finally, in Section~\ref{sec:setting} we introduce the changes in the setting that generalize straighforwardly in order to prove Theorem \ref{prop:main} in Section \ref{sec:proofD}.

	\section{A chamber basis for the Aomoto space}
	\label{sec:chamber-basis}
	
Let $v_1,\dots,v_n\in\Sp^{d-1}$ be unit vectors, no two of which are
	parallel. These vectors define the finite central hyperplane arrangement
	$$
	\cA=\{v_1^{\perp},\dots,v_n^{\perp}\},
	$$
	whose \emph{chambers} are the connected components of
	$\R^d\setminus\bigcup_{j=1}^{n}v_j^{\perp}$. The \emph{Aomoto space} of $\cA$ is defined as the finite-dimensional space of rational
	functions
	$$
	AO(\cA)
	=
	\Span\left\{\,
	\prod_{j\in S}\frac{1}{\langle v_j,x\rangle}
	\ :v_j\text{ linearly independent for every $S\subseteq\{1,\dots,n\}$}
	\,\right\}.
	$$
	As proved in \cite{MOM}, each chamber of $\cA$ contains exactly one point of the set of extremal points
	$$
	\mathscr{E}(P)
	=
	\left\{\,
	u\in\R^d
	\ :\
	u=\frac1n\sum_{j=1}^{n}\frac{v_j}{\langle v_j,u\rangle}
	\,\right\},
	$$
    where $P$ is defined by equation \eqref{eq:P}. It follows then that
	\begin{equation}\label{eq:OA}
	\lvert\mathscr{E}(P)\rvert
	=
	\#\{\text{chambers of }\cA\}
	=
	\dim AO(\cA),
    \end{equation}
	by the dimension theorem of Orlik and Terao \ref{thm:orlikterao} (cf. \cite{orlik-terao}). 	
    
    In this paper we exhibit an isomorphism of $AO(\cA)$ with $\mathbb{R}^{|\mathscr{E}(P)|}$, together
	with the interpolation formula it induces. The construction rests on two objects
	attached to the hyperplane arragement $\cA$.
	
	It is known that the local extrema of the polynomial $P$ on $\Sp^{d-1}$ at which $P$ does not vanish are, automatically,  unit vectors. 
    
    Let us check that the functions $L_u$ lie in the Aomoto space.
	
	\begin{lemma}\label{lem:membership}
		For every $u\in\mathscr{E}(P)$ we have $L_u\in AO(\cA)$.
	\end{lemma}
	
	\begin{proof}
		For arbitrary scalars $t_1,\dots,t_n$, a determinant expansion gives
$$\det\left(I+\sum_{j=1}^{n}t_j\,v_j\otimes v_j\right)
		=
		\sum_{S\subseteq\{1,\dots,n\}}
		\det\bigl(\langle v_j,v_k\rangle\bigr)_{j,k\in S}
		\prod_{j\in S}t_j .
		$$
        Applying this with
		$$t_j=\frac{1}{n\,\langle v_j,u\rangle\langle v_j,x\rangle}$$ 
    concludes the proof if the products are reciprocals of $\langle v_i,x\rangle$ with linearly independent $v_i$. Note that the  Gram determinant $\det(\langle v_j,v_k\rangle)_{j,k\in S}$ vanishes whenever the vectors $(v_j)_{j\in S}$ are linearly dependent.
	\end{proof}
	
	The defining equations of the critical points make the functions $L_u$ behave
	like Lagrange cardinal functions.
	
	\begin{proposition}[Lagrange property]\label{prop:lagrange}
		For all $u,w\in\mathscr{E}(P)$,
		$$
		L_u(w)=\delta_{u,w}.
		$$
	\end{proposition}
	
	\begin{proof}
    Let us observe first that if $w=u$ the definition of $\mu(u)$ gives $$L_u(u)=\mu(u)\,\mu(u)^{-1}=1.$$ To prove complete the proof let us substract the defining equations of $u$ and $w$ gives
		\begin{equation}
		    \label{eq:u-w}
            u-w
		=
		\frac1n\sum_{j=1}^{n}
		\left(\frac{1}{\langle v_j,u\rangle}-\frac{1}{\langle v_j,w\rangle}\right)v_j .
		\end{equation}
		We will rewrite it using
		$$
		\frac{1}{\langle v_j,u\rangle}-\frac{1}{\langle v_j,w\rangle}
		=
		-\frac{\langle v_j,u-w\rangle}{\langle v_j,u\rangle\langle v_j,w\rangle}$$
		Indeed, it is now evident that  \eqref{eq:u-w} is equivalent to
        $$
		\left(
		I+\frac1n\sum_{j=1}^{n}
		\frac{v_j\otimes v_j}{\langle v_j,u\rangle\langle v_j,w\rangle}
		\right)(u-w)=0 .
		$$
		If $u\neq w$ are different, the vector $u-w$ is a nonzero element of the kernel of the
		operator in the parentheses, which is therefore singular, which yields the vanishing $L_u(w)=0$. 
	\end{proof}
	
	\begin{theorem}[Interpolation formula]\label{thm:basis}
		The family $\{L_u:u\in\mathscr{E}(P)\}$ is a basis of $AO(\cA)$. Moreover, every
		$f\in AO(\cA)$ we can express it as
		$$
		f(x)=\sum_{u\in\mathscr{E}(P)}f(u)\,L_u(x).
		$$
	\end{theorem}
	
	\begin{proof}
		The functions $L_u$ are linearly independent: if $$\sum_{u\in\mathscr{E}(P)}c_u L_u=0,$$
        then
		evaluating at any $w\in\mathscr{E}(P)$ and using
		Proposition~\ref{prop:lagrange} gives $c_w=0$. By
		identity \ref{eq:OA} they form a basis. Given
		$f\in AO(\cA)$, write $f=\sum_u c_u L_u$; evaluating at $w$ and using
		$L_u(w)=\delta_{u,w}$ gives $c_w=f(w)$, which is the stated formula.
	\end{proof}
	
	The interpolation formula reduces identities between rational functions in
	$AO(\cA)$ to statements about their values at the finitely many critical points.
	We record the identities obtained by interpolating the constant function and the
	individual reciprocals $1/\langle v_i,x\rangle$.
	
	\begin{theorem}[Chamber identities]\label{thm:chamber-identities}
		The weights $\mu(u)$, $u\in\mathscr{E}(P)$, satisfy
		$$
		\sum_{u\in\mathscr{E}(P)}\mu(u)=1,\quad 
		\sum_{u\in\mathscr{E}(P)}\frac{\mu(u)}{\langle v_i,u\rangle}=0
		\qquad\text{and}\qquad
		\sum_{u\in\mathscr{E}(P)}
		\frac{\mu(u)}{\langle v_i,u\rangle\langle v_j,u\rangle}
		=
		n\,\delta_{ij}
		$$
        for all $1\le i,j\le n$. 
        \end{theorem}

  The first chamber identity corresponds to Theorem 1.3 as in 
Aomoto and Forrester \cite{AomotoForrester}. This connection is not a suprise to the authors as they already dug into the residue formulas behind its proof while working towards the completion of \cite{MOM}. It is a surprise, thought, that they were not able to locate it back then. This connection also seems to have been unintentionally overlooked by Ouimet and Greaves in \cite{OUI}.

	\begin{proof}
		Expanding the determinant in the definition of $L_u$ as in the proof of
		Lemma~\ref{lem:membership} and grouping the terms by reciprocal degree,
		$$
		L_u(x)
		=
		\mu(u)
		\left(
		1+\frac1n\sum_{j=1}^{n}\frac{1}{\langle v_j,u\rangle\langle v_j,x\rangle}
		+R_u(x)
		\right),
		$$
		where $R_u$ collects the reciprocal products of degree at least two. We first justify that the relevant coefficients may be compared. Suppose
	that
	\[
F_0(x)+F_1(x)+\cdots+F_d(x)=0,
	\]
	where $F_k(sx)=s^{-k}F_k(x)$ for every $s>0$. For a fixed $x$ outside the
	hyperplanes, replacing $x$ by $sx$ gives a polynomial identity in $s^{-1}$;
	hence $F_k(x)=0$ for every $k$. Moreover, the functions
	\[
		\frac1{\ip{v_1}{x}},\ldots,
		\frac1{\ip{v_n}{x}}
	\]
	are linearly independent. Indeed, if
	\[
		\sum_{j=1}^n\frac{a_j}{\ip{v_j}{x}}=0,
	\]
	then multiplication by $\prod_{k=1}^n\ip{v_k}{x}$ gives a polynomial
	identity. For a fixed $i$, restrict this identity to a point of
	\[
		v_i^\perp\setminus\bigcup_{j\neq i}v_j^\perp.
	\]
	Such a point exists because the hyperplanes are distinct. All terms vanish except the one containing $a_i$, 	therefore $a_i=0$. Since $i$ is arbitrary, the claim follows.

		Applying the interpolation formula of Theorem~\ref{thm:basis} to $f=1$ and
		comparing the terms of reciprocal degree zero gives $\sum_{u}\mu(u)=1$.
		
		Fix $1\le i\le n$ and apply the same argument to $f(x)=1/\langle v_i,x\rangle$, so that
		$$
		\frac{1}{\langle v_i,x\rangle}
		=
		\sum_{u\in\mathscr{E}(P)}\frac{1}{\langle v_i,u\rangle}L_u(x).
		$$
		The terms of reciprocal degree zero give
		$\sum_{u}\mu(u)/\langle v_i,u\rangle=0$, while those of reciprocal degree one
		give
		$$
		\frac{1}{\langle v_i,x\rangle}
		=
		\frac1n\sum_{j=1}^{n}
		\left(
		\sum_{u\in\mathscr{E}(P)}
		\frac{\mu(u)}{\langle v_i,u\rangle\langle v_j,u\rangle}
		\right)
		\frac{1}{\langle v_j,x\rangle}.
		$$
		By linear independence of the $1/\langle v_j,x\rangle$ yields
		$\sum_{u}\mu(u)/(\langle v_i,u\rangle\langle v_j,u\rangle)=n\,\delta_{ij}$.
        \end{proof}

The same method provides more identities, including one that was originally proved with the aid of the Euler-Jacobi vanishing theorem in \cite{MOM} and used there to prove several polarization inequalities.
        \begin{theorem}    \label{thm:weigthed}        
        With our previous notation the following identities hold:
		$$
		\sum_{u\in\mathscr{E}(P)}\mu(u)\,u=0
		\qquad\text{and}\qquad
		\sum_{u\in\mathscr{E}(P)}\mu(u)\,u\otimes u
		=
		\frac1n\sum_{j=1}^{n}v_j\otimes v_j .
		$$
		Furthermore,
		\begin{equation}\label{eq:local-global}
			\sum_{u\in\mathscr{E}(P)}
			\left(\sum_{j=1}^{n}\frac{1}{\langle v_j,u\rangle^{2}}-n^{2}\right)\mu(u)=0 .
		\end{equation}
	\end{theorem}

        \begin{proof}
		The equation  that defines $\mathscr{E}(P)$, namely
        $$u=\frac1n\sum_{j=1}^n\frac{v_j}{\langle v_j,u\rangle}$$ provides more identities. Indeed,
		$$
		\sum_{u\in\mathscr{E}(P)}\mu(u)\,u
		=
		\frac1n\sum_{j=1}^{n}v_j
		\sum_{u\in\mathscr{E}(P)}\frac{\mu(u)}{\langle v_j,u\rangle}
		=
		0,
		$$
		by the second chamber identity in Theorem \ref{thm:chamber-identities}.
        Similarly,
		$$
		\sum_{u\in\mathscr{E}(P)}\mu(u)\,u\otimes u
		=
		\frac{1}{n^{2}}\sum_{j,k=1}^{n}v_j\otimes v_k
		\sum_{u\in\mathscr{E}(P)}\frac{\mu(u)}{\langle v_j,u\rangle\langle v_k,u\rangle}
		=
		\frac1n\sum_{j=1}^{n}v_j\otimes v_j 
		$$
		by the third chamber identity in Theorem \ref{thm:chamber-identities}. Finally, putting $i=j$ in the third chamber identity in Theorem \ref{thm:chamber-identities} and summing over $j$
		gives 
        $$
        \sum_{u\in \mathscr{E}(P)}\mu(u)\sum_{j=1}^n\frac{1}{\langle v_j,u\rangle^{2}}=n^{2},
        $$
        which together
		with the first chamber identity,  $\sum_{u}\mu(u)=1$, yields \eqref{eq:local-global}.
	\end{proof}


    \section{Proof of Theorem \ref{thm:introB}}
	\label{sec:equality}
	
	We now use the chamber basis to characterize the extremizers of the strong
	polarization inequality. By Lemma 7.6 in~\cite{MOM} the vectors of extremal configurations are non parallel for any extremal configuration. This allows us to reduce our proof to the following

    \begin{proposition}[Harmonicity of extremizers]\label{thm:harmonicity}
		Let $v_1,\dots,v_n\in\Sp^{d-1}$ be unit vectors, no two of which are parallel,
		and set
		$$
		P(x)=\prod_{j=1}^{n}\langle v_j,x\rangle.
		$$
		If $v_1,\dots,v_n$ is an extremal configuration for the strong polarization
		inequality, then $P$ is harmonic.
	\end{proposition}
	
	\begin{proof}
		Write
		$$
		Q(x)=\sum_{j=1}^{n}\frac{1}{\langle v_j,x\rangle^{2}},
		\qquad
		x\in\Sp^{d-1}\setminus\{P=0\}.
		$$
		Extremality means that $Q(x)\ge n^{2}$ for every
		$x\in\Sp^{d-1}\setminus\{P=0\}$; in particular, $Q(u)\ge n^{2}$ for each
		$u\in\mathscr{E}(P)$.
		
		By the local--global identity \eqref{eq:local-global},
		$$
		\sum_{u\in\mathscr{E}(P)}\bigl(Q(u)-n^{2}\bigr)\mu(u)=0 .
		$$
		Every summand is nonnegative and every weight $\mu(u)$ is positive, so
		$$
		Q(u)=n^{2}\qquad\text{for all }u\in\mathscr{E}(P).
		$$
	The Laplacian $\Delta P$ can be computed by logarithmic differentiation (cf. Lemma 7.5 in \cite{MOM}). For any $x\notin\{P=0\}$,
		$$
		\frac{\Delta P(x)}{P(x)}
		=
		\norm{\sum_{j=1}^{n}\frac{v_j}{\langle v_j,x\rangle}}^{2}-Q(x).
		$$
		Fix an extremizer $u\in\mathscr{E}(P)$, by definition it satisfies $$\sum_{j=1}^{n}\frac{v_j}{\langle v_j,u\rangle}=nu.$$ Since
		$\norm{u}=1$, the norm term equals $\norm{nu}^{2}=n^{2}$, and therefore
		$$
		\frac{\Delta P(u)}{P(u)}=n^{2}-Q(u)=0 
		$$
		at any extremal point $u\in \mathscr{E}(P)$.
        
		On the other hand, expanding the squared norm and subtracting $Q(x)$ gives the equivalent expression
		$$
		\frac{\Delta P(x)}{P(x)}
		=
		2\!\!\sum_{1\le j<k\le n}
		\frac{\langle v_j,v_k\rangle}{\langle v_j,x\rangle\langle v_k,x\rangle}.
		$$
		Because no two of the $v_j$ are parallel, each pair $(v_j,v_k)$ is linearly
		independent. As a consequence the quotient $\Delta P/P$ is a linear combination of reciprocal products over
		independent pairs. In other words,  $\Delta P/P\in AO(\cA)$. Applying the interpolation
		formula of Theorem~\ref{thm:basis} and using $\Delta P(u)/P(u)=0$ at every
		extremal point $u$ yields
		$$
		\frac{\Delta P(x)}{P(x)}
		=
		\sum_{u\in\mathscr{E}(P)}\frac{\Delta P(u)}{P(u)}\,L_u(x)
		=
		0 .
		$$
		Thus $\Delta P$ vanishes on the complement of the hyperplanes; being a
		polynomial, it vanishes identically. Therefore $P$ is harmonic.
	\end{proof}
	
\section{The generalized setting}\label{sec:setting}

In this section we will introduce several dummy variables that will be useful later in the proof. For the most part they represent straightforward generalizations of our previous observations. Fix positive parameters $t_1,\ldots,t_n$ and unit vectors
$v_1,\ldots,v_n\in\R^d$ and  define the map
\begin{equation}\label{eq:T-map}
 T(x)
 =x-
 \sum_{j=1}^n
 \frac{t_jv_j}{\ip{v_j}{x}}.
\end{equation}
For $y\in\R^d$, the corresponding shifted extremal set is
\begin{equation}\label{eq:shifted-extremals}
 \mathscr{E}_{y}(P;\mathbf t)
 =\left\{
 u\in\mathbb{R}^d:
 u
 =y+
 \sum_{j=1}^n
 \frac{t_jv_j}{\ip{v_j}{u}} \textrm{ such that $P(u)\neq 0$}
 \right\}.
\end{equation}
Thus $\mathscr{E}_{y}(P;\mathbf t)=T^{-1}(y)$. In the equal-parameter normalization $t_1=\cdots=t_n=1/n$ with $y=0$ we recover the extremal points previously defined. The following result might also be traced back to the work of Aomoto and Forrester (see also  \cite[Section~3, equation~(20)]{CaseiroFrancoiseSasaki}).

\begin{lemma}\label{lem:branches}
Fix a chamber $C$. Then, for every $y\in\R^d$ the equation
$T(u)=y$ has a unique solution $u_C(y)$ in $C$.
Moreover, the restriction map $T|_C:C\to\R^d$ is a smooth diffeomorphism and
\begin{equation}\label{eq:DT}
 DT(u)
 =I+
 \sum_{j=1}^n
 \frac{t_j\,v_j\otimes v_j}
      {\ip{v_j}{u}^{2}}.
\end{equation}
\end{lemma}

The reader might find their own proof following closely the one we gave in \cite{MOM} which is paralleled in \cite{OUI}. The previous proposition suggests we adapt our the definition of $\mu$ to the following
\begin{equation}\label{eq:mu-general}
 \mu(u)
 =
 \det\left(
 I+
 \sum_{j=1}^n
 \frac{t_j\,v_j\otimes v_j}
      {\ip{v_j}{u}^{2}}
 \right)^{-1}.
\end{equation}

Similarly, define
\begin{equation}\label{eq:L-general}
 L_{u}(x)
 =\mu(u)
 \det\left(
 I+
 \sum_{j=1}^n
 \frac{t_j\,v_j\otimes v_j}
 {\ip{v_j}{u}\ip{v_j}{x}}
 \right).
\end{equation}
whenever the denominators do not vanish.  Again, the functions $L_{u}$ consequently form have the interpolation property as a consequence of the following

\begin{lemma}[Lagrange property]\label{lem:lagrange2}
	For all $u,x\in\mathscr{E}_{y}(P;\mathbf t)$,
	\[
		L_{u}(x)
		=
		\delta_{u,x}.
	\]
\end{lemma}

A proof following the same ideas of our previous is left to the reader. We now state the chamber identities in the form needed below. 

\begin{lemma}[Generalized Chamber identities]\label{thm:chamber-identities2}
Let $v_1,\ldots,v_n$ be unit vectors span $\R^d$ with no collinear pair. For every $y\in\R^d$, the
weights $\mu(u)$, $u\in\mathscr{E}_{y}(P;\mathbf t)$, satisfy
\begin{align}
 \sum_{u\in\mathscr{E}_{y}(P;\mathbf t)}\mu(u)
 &=1,\label{eq:chamber-0}\\
 \sum_{u\in\mathscr{E}_{y}(P;\mathbf t)}
 \frac{\mu(u)}{\ip{v_i}{u}}
 &=0,\qquad 1\leq i\leq n,\label{eq:chamber-1}\\
 \sum_{u\in\mathscr{E}_{y}(P;\mathbf t)}
 \frac{\mu(u)}
 {\ip{v_i}{u}\ip{v_j}{u}}
 &=\frac{\delta_{ij}}{\sqrt{t_it_j}},
 \qquad 1\leq i,j\leq n.
 \label{eq:chamber-2}
\end{align}

\end{lemma}

In particular, if $t_1=\cdots=t_n=1/n$, we recover the chamber identities in Theorem \ref{thm:chamber-identities}. Their proof is omitted as they follow by straightforward generalization of our previous arugments.

\subsection{The change of coordinates}

In this section we follow \cite{OUI}. First we introduce an identity that will be used to identify a probability measure that plays a relevant r\^ole in the proof. 

\begin{lemma}[Quadratic identity]\label{lem:quadratic}
For every $x$ such that $P(x)\neq 0$, the following identity
\begin{equation}\label{eq:quadratic}
 \|x\|^2+
 \left\|
 \sum_{j=1}^n
 \frac{t_jv_j}{\ip{v_j}{x}}
 \right\|^2
 =\|T(x)\|^2+2\sum_{j=1}^n t_j.
\end{equation}
holds.
\end{lemma}

\begin{proof}
it follows after expanding the square in the definition of $T$ above (equation \eqref{eq:T-map}) and taking into account the identity
\[
 \left\langle
 x,
 \sum_{j=1}^n
 \frac{t_jv_j}{\ip{v_j}{x}}
 \right\rangle
 =\sum_{j=1}^n t_j.
\]
\end{proof}

Let $d\gamma_d$ denote standard Gaussian measure on $\R^d$, namely
\[
 d\gamma_d(x)
 =(2\pi)^{-d/2}e^{-\|x\|^2/2}\,\dd x.
\]
The integrand on equation \eqref{eq:chamberwise-change} in the following result captures the importance of chamber identities in connection to integration against Gaussian measures. 
	
	\begin{lemma}[Chamberwise change of variables]\label{lem:cov}
		Under the assumptions of Lemma~\ref{thm:chamber-identities2}, let $\dd\nu$ be the Borel
		measure on $\mathbb R^d$ defined by
		\begin{equation}\label{eq:nu}
			d\nu(x)
			=\exp\left(
			\sum_{j=1}^n t_j-
			\frac12\left\|
			\sum_{j=1}^n
			\frac{t_jv_j}{\ip{v_j}{x}}
			\right\|^2
			\right)d\gamma_d(x).
		\end{equation}
		Then for every $F\in L^1(\dd\nu)$,
		\begin{equation}\label{eq:chamberwise-change}
			\int_{\mathbb R^d}F(x)\,d\nu(x)
			=\int_{\mathbb R^d}\!\!\sum_{u\in\mathscr E_y(P;\mathbf t)}\!\!\mu(u)\,F(u)\,d\gamma_d(y).		\end{equation}
	\end{lemma}
	
	\begin{proof}
		The union of the hyperplanes $\{P=0\}$ has Gaussian measure zero and may be discarded
		from every integral. Hence,
		decomposing $\mathbb R^d\setminus\{P=0\}$ into its chambers $C$ yields
		\[
		\int_{\mathbb R^d}F\,d\nu
		=(2\pi)^{-d/2}\int_{\mathbb R^d}F(x)\,e^{-\|T(x)\|^2/2}\,\dd x
		=(2\pi)^{-d/2}\sum_C\int_C F(x)\,e^{-\|T(x)\|^2/2}\,\dd x .
		\]
		By Lemma~\ref{lem:branches}, for each chamber $C$ the restriction
		$T|_C\colon C\to\mathbb R^d$ is a diffeomorphism onto $\mathbb R^d$ with inverse $u_C$,
		and the substitution $y=T(x)$, $x=u_C(y)$, has Jacobian $\dd x=\mu(u_C(y))\,\dd y$.
		Therefore
		\[
		\int_C F(x)\,e^{-\|T(x)\|^2/2}\,\dd x
		=\int_{\mathbb R^d}F\bigl(u_C(y)\bigr)\,\mu\bigl(u_C(y)\bigr)\,e^{-\|y\|^2/2}\,\dd y .
		\]
	 Adding over all the chambers one gets
		\[
		\int_{\mathbb R^d}F\,d\nu
		=(2\pi)^{-d/2}\int_{\mathbb R^d}\!\!\sum_{u\in\mathscr E_y(P;\mathbf t)}\!\!
		\mu(u)\,F(u)\,e^{-\|y\|^2/2}\,\dd y
		=\int_{\mathbb R^d}\!\!\sum_{u\in\mathscr E_y(P;\mathbf t)}\!\!\mu(u)\,F(u)\,d\gamma_d(y),
		\]
		which is \eqref{eq:chamberwise-change}. 
	\end{proof}
	
	\begin{lemma}\label{lem:tilt}
		Under the assumptions of Lemma~\ref{thm:chamber-identities2}, the measure $\nu$ defined
		in \eqref{eq:nu} is a probability measure. Furthermore,
		\begin{align}
			\E_\nu\left(\frac1{\ip{v_i}{\mathbf Z}}\right)&=0,
			\label{eq:nu-first}\\
			\E_\nu\left(\frac1
			{\ip{v_i}{\mathbf Z}\ip{v_j}{\mathbf Z}}\right)
			&=\frac{\delta_{ij}}{\sqrt{t_it_j}},
			\label{eq:nu-second}
		\end{align}
		the expectations being computed with respect to $d\nu(\mathbf Z)$.
	\end{lemma}
	
	\begin{proof}
		Apply Lemma~\ref{lem:cov} to $F\equiv1$. Since
		$\sum_{u\in\mathscr E_y(P;\mathbf t)}\mu(u)=1$ by \eqref{eq:chamber-0} and
		$\gamma_d$ is a probability measure,
		\begin{equation}\label{eq:normalization-change}
			\E_\nu(1)
			=\int_{\mathbb R^d}\!\!\sum_{u\in\mathscr E_y(P;\mathbf t)}\!\!\mu(u)\,d\gamma_d(y)
			=\int_{\mathbb R^d}1\,d\gamma_d(y)=1,
		\end{equation}
		so $\nu$ is a probability measure.
		
		Fix $i$ and apply the nonnegative case of \eqref{eq:chamberwise-change} we apply the same argument to
		$F(x)=\ip{v_i}{x}^{-2}$. By the diagonal case of \eqref{eq:chamber-2},
		\[
		\E_\nu\!\left(\frac{1}{\ip{v_i}{\mathbf Z}^2}\right)
		=\int_{\mathbb R^d}\!\!\sum_{u\in\mathscr E_y(P;\mathbf t)}\!\!
		\frac{\mu(u)}{\ip{v_i}{u}^2}\,d\gamma_d(y)
		=\frac1{t_i}<\infty .
		\]
		Thus $\ip{v_i}{\mathbf Z}^{-1}\in L^2(\nu)$, and since $\nu$ is a probability measure
		$L^2(\nu)\subseteq L^1(\nu)$; by Cauchy--Schwarz both $\ip{v_i}{\cdot}^{-1}$ and
		$\ip{v_i}{\cdot}^{-1}\ip{v_j}{\cdot}^{-1}$ lie in $L^1(\nu)$. Applying the $L^1$ case of
		\eqref{eq:chamberwise-change} to $F(x)=\ip{v_i}{x}^{-1}$ with \eqref{eq:chamber-1} yields
		\eqref{eq:nu-first}, and to $F(x)=\ip{v_i}{x}^{-1}\ip{v_j}{x}^{-1}$ with
		\eqref{eq:chamber-2} yields \eqref{eq:nu-second}.
	\end{proof}

\section{Proof of Theorem~\ref{prop:main}}\label{sec:proofD}

   Let us introduce momentarily a class of functions that correspond to the completely monotone ones in Theorem \ref{prop:main} above. Given $d\rho$ a ($\sigma$-finite Borel) measure on 
$[0,\infty)$ such that
\begin{equation}  \label{eq:finiteness}\int_0^{\infty}\exp(-t)d\rho(t)
\end{equation}
is finite. We define its transform
\begin{equation}\label{eq:fj-representation}
 g(r)
 =\int_0^{\infty}
 \exp\!\left(-\frac{t^2}{2r^2}\right)d\rho(t).
\end{equation}
We denote the class of functions that arise as such transforms by $\mathcal{G}$. This transform is closely related to the Laplace transform by a change of variables. In particular, by the Bernstein-Hausdorff-Widder theorem it is closely related to completely monotonic functions and might be written as Laplace transforms of measures of the form $\tilde{\rho}(t)\frac{dt}{\sqrt{t}}$ (cf. Widder \cite{Widder1946}).  This class identifies a useful form to which the functions $f_j$ in Theorem~\ref{prop:main} above translates. The finiteness restriction on the integral \eqref{eq:finiteness} might be avoided by a limiting process in the $f_j$. Indeed, each $f_j$ is nonzero and completely monotone, Bernstein-Hausdorff-Widder's theorem implies that $f_j(s)>0$ for every $s>0$. Hence
$$
\mathbb E_Z f_j\left(\frac{1}{Z^2}\right)\in(0,\infty].
$$
The inequality is therefore meaningful in the extended positive real numbers. It follows by applying the finite-valued inequality to bounded monotone approximations of the functions $f_j$ and then using the monotone convergence theorem.

Let us state it first

\begin{proposition}\label{prop:5.1}
Let $v_1,\ldots,v_n\in\mathbb{S}^{d-1}$ no two of which are collinear,  $\mathbf Z$ be the standard Gaussian vector and $g_j\in\mathcal{G}$  for any
 $j=1,\ldots,n$.
Then
\begin{equation}\label{eq:main-inequality}
 \E_{\mathbf Z}\left(\prod_{j=1}^n
 g_j\!\left(\left|\ip{v_j}{\mathbf Z}\right|\right)\right)
 \geq
 \prod_{j=1}^n \E_Z g_j(|Z|).
\end{equation}
\end{proposition}

Some remarks are in order: the extra collinearity hypothesis is there because we invoke Lemma \ref{thm:chamber-identities2} in its proof. It is clear that one might perturb the $v_j$ variables and then take a limit.  Proposition \ref{prop:5.1} allows to assume the integrals involved in the proof exist and avoids distinguishing cases when the inequalities in the statement of Theorem \ref{prop:main} diverge.


\begin{corollary}\label{corr}
Let $v_1,\ldots,v_n\in\mathbb{S}^{d-1}$, $a_1,\ldots,a_n>0$  and $\mathbf Z$ be the standard Gaussian vector. Then
\[
 \E_{\mathbf{Z}}\left(\prod_{j=1}^n
 \log\left(1+\frac{\ip{\mathbf v_j}{\mathbf Z}^2}{a_j}\right)\right)
 \geq
 \prod_{j=1}^n
 \E_Z\log\left(1+\frac{Z^2}{a_j}\right)
\]
holds.
\end{corollary}
This is a starightforward application of Theorem \ref{prop:main} and the identity
\begin{equation}
 \log\left(1+\frac{r^2}{a}\right)
 =\int_0^\infty
 e^{-u/r^2}\frac{1-e^{-u/a}}{u}\dd u,
\end{equation}
which follows for any $a>0$ by Frullani's identity.

\subsection{Proof of Proposition \ref{prop:5.1}}
Before we start it will be useful to study the terms on the right-hand side of \eqref{eq:main-inequality}

\begin{lemma}\label{lem:onedim}
	For every $t\geq0$,
	\begin{equation}\label{eq:one-dimensional-kernel}
		\E_Z\left[
		\exp\left(-\frac{t^2}{2Z^2}\right)
		\right]
		=
		e^{-t}.
	\end{equation}
\end{lemma}

We leave the proof to the reader.

We are now ready to study the inequality in the statement. Let us start rewritting the right hand side as follows. Using the definition of the $g_j$ in equation \eqref{eq:fj-representation}, \eqref{eq:one-dimensional-kernel}, Lemma \ref{lem:onedim}
and Tonelli's theorem, we obtain
\begin{align}
	\prod_{j=1}^n\E_Z g_j(|Z|)
	&=
	\prod_{j=1}^n
	\int_{0}^{\infty}
	\E_Z\exp\left(-\frac{t^2}{2Z^2}\right)\dd
	\rho_j( t)
	\notag\\
	&=
	\prod_{j=1}^n
	\int_{0}^{\infty}e^{-t}\dd\rho_j( t)
	\notag\\
	&=
	\int_{[0,\infty)^n}
	\exp\left(-\sum_{j=1}^n t_j\right)
	\prod_{j=1}^n\dd\rho_j( t_j).
	\label{eq:right-kernel-representation}
\end{align}

The integral representations \eqref{eq:fj-representation} and Tonelli's
theorem give the following expression for the left hand side
\begin{align}
	\E_{\mathbf Z}\left(\prod_{j=1}^n
	g_j\!\left(\left|\ip{v_j}{\mathbf Z}\right|\right)\right)
	&=
	\int_{[0,\infty)^n}
	\E_{\mathbf Z}\exp\left(
	-\frac12\sum_{j=1}^n
	\frac{t_j^2}{\ip{v_j}{\mathbf Z}^2}
	\right)
	\prod_{j=1}^n\dd\rho_j( t_j)
	\notag\\
	&=
	\int_{[0,\infty)^n}
	K(\mathbf t)
	\prod_{j=1}^n\dd\rho_j( t_j).
	\label{eq:left-kernel-representation}
\end{align}
where
\begin{equation}\label{eq:kernel-c}
	K(\mathbf t)
	=
	\E_{\mathbf Z}\exp\left(
	-\frac12\sum_{j=1}^n
	\frac{t_j^2}{\ip{v_j}{\mathbf Z}^2}
	\right).
\end{equation}
for any $\mathbf t=(t_1,\ldots,t_n)\in[0,\infty)^n$. 

Comparison of
\eqref{eq:left-kernel-representation} and
\eqref{eq:right-kernel-representation} shows that Theorem \ref{prop:main} follows
from the following pointwise estimate.

\begin{lemma}
	\label{lem:reciprocal-kernel}
For every
	$\mathbf t=(t_1,\ldots,t_n)\in[0,\infty)^n$,
	\begin{equation}\label{eq:kernel-bound}
		K(\mathbf t)
		\geq
		\exp\left(-\sum_{j=1}^n t_j\right).
	\end{equation}
\end{lemma}

Before we conclude let us say that this estimate is one of the key elements of the proof by Ouimet and Greaves \cite[proof of Theorem~1, equation~(9)]{OUI}. It should be noted that it is equivalent to Theorem \ref{prop:main} with $f_j(x)=\exp(-\frac{1}{2}t_j^2x)$.

Assuming the lemma
for the moment, equations \eqref{eq:left-kernel-representation} and
\eqref{eq:kernel-bound} give
\[
	\E_{\mathbf Z}\left(\prod_{j=1}^n
	g_j\!\left(\left|\ip{v_j}{\mathbf Z}\right|\right)\right)
	\geq
	\int_{[0,\infty)^n}
	\exp\left(-\sum_{j=1}^n t_j\right)
	\prod_{j=1}^n\dd\rho_j( t_j).
\]
By \eqref{eq:right-kernel-representation}, the last integral equals
$\prod_{j=1}^n\E_Z g_j(|Z|)$. This proves
\eqref{eq:main-inequality}.

\begin{proof}[Proof of Lemma~\ref{lem:reciprocal-kernel}]
	Fix
	$\mathbf t=(t_1,\ldots,t_n)\in(0,\infty)^n$. The vectors
	$v_1,\ldots,v_n$ are non collinear pairwise and, consequently,
	Lemma~\ref{lem:tilt} applies with parameters $t_1,\ldots,t_n$.
	Let $\nu$ be the probability measure defined in \eqref{eq:nu}, and let
	$\mathbf Z$ have law $\nu$. Changing from standard Gaussian measure to
	$\nu$ gives
	\begin{align}
		K_R(\mathbf t)
		&=
		\exp\left(-\sum_{j=1}^n t_j\right)
		\E_\nu\exp\Bigg[
		-\frac12\Bigg\{
		\sum_{j=1}^n
		\frac{t_j^2}{\ip{v_j}{\mathbf Z}^2}
		-
		\left\|
		\sum_{j=1}^n
		\frac{t_jv_j}{\ip{v_j}{\mathbf Z}}
		\right\|^2
		\Bigg\}\Bigg].
		\label{eq:kernel-nu}
	\end{align}
	By \eqref{eq:nu-second},
	\begin{align}
		\E_\nu
		\left(\sum_{j=1}^n
		\frac{t_j^2}{\ip{v_j}{\mathbf Z}^2}\right)
		&=
		\sum_{j=1}^n
		t_j^2
		\E_\nu
		\frac{1}{\ip{v_j}{\mathbf Z}^2}
		=
		\sum_{j=1}^n t_j,
		\label{eq:first-energy}
        \end{align}

        \begin{align}
		\E_\nu
		\left\|
		\sum_{j=1}^n
		\frac{t_jv_j}{\ip{v_j}{\mathbf Z}}
		\right\|^2
		&=
		\sum_{i,j=1}^n
		t_it_j\ip{v_i}{v_j}
		\E_\nu\left(
		\frac{1}{
		\ip{v_i}{\mathbf Z}\ip{v_j}{\mathbf Z}}\right)
		\notag\\
		&=
		\sum_{j=1}^n t_j,
		\label{eq:second-energy}
	\end{align}
	where the last equality also uses $\|v_j\|=1$. Hence the random variable
	between braces in \eqref{eq:kernel-nu} has mean zero. Since
	$s\mapsto e^{-s/2}$ is convex, Jensen's inequality gives
	\[
		K(\mathbf t)
		\geq
		\exp\left(-\sum_{j=1}^n t_j\right).
	\]
\end{proof}

\textsc{Remark:} Notice that the kernel depends on $R$, namely
the correlation of the variables in the denominator, which is the Gram matrix
\[
	R=\bigl(\langle v_i, v_j\rangle\bigr)_{i,j=1}^n.
\]
In other words, the joint distribution of
\[
	\bigl(
	\ip{v_1}{\mathbf Z},\ldots,
	\ip{v_n}{\mathbf Z}
	\bigr)
\]
is $N(0,R)$ and hence depends only on the Gram matrix $R$. This dependence have been waved in our presentation as the correlation matrix plays no role is encrypted in the vectors $v_j$.

	\section*{Acknowledgements}
	
	This paper was written while the first named author was visiting Bruno Staffa at  the Max Planck Institute for Mathematics (at Bonn). He is grateful to them for their support and hospitality.  

    The authors acknowledge the use of Artificial Intelligence in the form of interaction with Large Language Models both in the research and drafting of this manuscript.
    
	\bibliographystyle{abbrv}
	\bibliography{references}
	\vspace{.5em}
	
	\noindent
	\begin{minipage}[t]{0.48\textwidth}
		\noindent
		Ángel D. Martínez\\
		\textsc{Departamento de Matemáticas}\\
		\textsc{CUNEF Universidad}\\
		\textsc{Madrid, Spain}\\
		\textit{Email:} \texttt{angeld.martinez@cunef.edu}
	\end{minipage}
	\hfill
	\begin{minipage}[t]{0.48\textwidth}
		\noindent
		Oscar Ortega-Moreno\\
		\textsc{Departamento de Matemáticas}\\
		\textsc{CUNEF Universidad}\\
		\textsc{Madrid, Spain}\\
		\textit{Email:} \texttt{oscar.ortegamoreno@cunef.edu}
	\end{minipage}
	
\end{document}